\documentclass[graybox]{svmult}

\usepackage{mathptmx}       % selects Times Roman as basic font
\usepackage{helvet}         % selects Helvetica as sans-serif font
\usepackage{courier}        % selects Courier as typewriter font
\usepackage{type1cm}        % activate if the above 3 fonts are
\usepackage{amsmath}
\usepackage{amsfonts}
\usepackage{amssymb}
\usepackage{makeidx}         % allows index generation
\usepackage{graphicx}        % standard LaTeX graphics tool
\usepackage{multicol}        % used for the two-column index
\usepackage[bottom]{footmisc}% places footnotes at page bottom

\makeindex             % used for the subject index
\DeclareMathOperator{\tr}{tr}
\DeclareMathOperator{\Rea}{Re}

\begin{document}

\title*{On conditions for weak conservativeness of regularized explicit finite-difference schemes for 1D barotropic gas dynamics equations}
\titlerunning{On conditions for weak conservativeness of explicit finite-difference schemes}
\author{A. Zlotnik and T. Lomonosov}
\institute{A. Zlotnik \at National Research University Higher School of Economics, Myasnitskaya 20, 101000 Moscow, Russia \email{azlotnik@hse.ru}
\and T. Lomonosov \at National Research University Higher School of Economics, Myasnitskaya 20, 101000 Moscow, Russia \email{tlomonosov@hse.ru}}
\maketitle

\abstract{We consider explicit two-level three-point in space finite-difference sche\-mes for solving 1D barotropic gas dynamics equations.
The schemes are based on special quasi-gasdynamic and quasi-hydrodynamic regularizations of the system.
We linearize the schemes on a constant solution and
derive the von Neumann type necessary condition and a CFL type criterion (necessary and sufficient condition) for weak conservativeness in $L^2$ for the corresponding initial-value problem on the whole line.
The criterion is essentially narrower than the necessary condition and wider than a sufficient one obtained recently in a particular case; moreover, it corresponds most well to numerical results for the original gas dynamics system.}


\section{\large Introduction}
The stability theory for finite-difference schemes for model problems in gas dynamics is well presented in the literature \cite{B73,Cou13,GV96,GR73,GKO95,LeV04,RM72}.
In this paper we consider some finite-difference schemes for solving 1D barotropic gas dynamics equations.
The schemes are explicit, two-level in time and use a symmetric three-point stencil in space.
Their construction is based on special quasi-gasdynamic and quasi-hydrodynamic \cite{Ch04,E07,Sh09,Z08,ZCh08} regularizations of the original equations (without a regularization, the schemes are unstable).
The schemes of this kind were successfully applied in numerous and various practical applications, in particular, see
\cite{BZS16,EB11,EZI17,ZG16}, but their theory is not developed so well.
\par We linearize the schemes on a constant solution and derive both the von Neumann type necessary condition and a CFL type criterion for weak conservativeness in $L^2$ for the corresponding initial-value problem on the whole line.
The weak conservativeness in $L^2$ means the uniform in time bound for the norm of scaled solution by the norm of initial data instead of the energy conservation law for the linearized original system, i.e., the acoustics system of equations.
Our numerical experience show that validity of the weak conservativeness property is important since it prevents numerical solutions from the well-known possible spurious oscillations.
The property guarantees the uniform in time $L^2$ stability with respect to initial data.
\par Since in practice necessary conditions are often in use (a derivation of sufficient conditions is much more complicated in general), it is important to know to what extent this is lawful to do.
The criterion turns out to be essentially narrower than the necessary condition and at the same time wider than a sufficient condition obtained recently in a particular case in \cite{SSh13}.
Moreover, namely the criterion corresponds most well to results of numerical experiments for the original gas dynamics system.
Therefore the criterion (but not the necessary condition or sufficient one) is most adequate and useful for practical purposes.
\section{\large Systems of equations, finite-difference schemes and their linearization}
\par The 1D barotropic gas dynamics (Euler) system of equations consists in the mass and momentum balance equations
\begin{gather}
 \partial_t\rho+\partial_x(\rho u) = 0,\ \
 \partial_t(\rho u)+\partial_xp(\rho)=0,
\label{eq:bgd}
\end{gather}
where $\rho>0$, $u$ and $p$ are the gas density and velocity (the sought functions) and pressure.
We assume that $p'(\rho)>0$ and consider the equations for $x\in\mathbb{R}$ and $t>0$.
\par The 1D barotropic quasi-gas dynamics (QGD) system of equations consists in the regularized mass and momentum balance equations
\begin{gather}
\partial_t\rho+\partial_xj = 0,\ \
\partial_t(\rho u)+\partial_x\big(ju+p(\rho)-\Pi\big)=0,
\label{eq:qgd1}\\[1mm]
j=\rho(u-w),\ \
w = \frac{\tau}{\rho}u\partial_x(\rho u)+\hat w,\ \
\hat w = \frac{\tau}{\rho}[\rho u\partial_xu + p'(\rho)],
\label{eq:qgd2}\\[1mm]
\Pi = \Pi_{NS} + \rho u\hat w + \tau p'(\rho)\partial_x(\rho u),\ \ \Pi_{NS}=\mu(\rho)\partial_xu.
\label{eq:qgd3}
\end{gather}
Here $j$ and $\Pi$ are the regularized mass flux and stress,
$w$ and $\hat w$ are the regularizing velocities,
$\tau=\tau(\rho)>0$ is a regularization parameter
and $\Pi_{NS}$ is the Navier-Stokes viscous stress with $\mu(\rho)\geq 0$ being proportional to the viscosity coefficient.
In the barotropic case, quasi-gasdynamic and quasi-hydrodynamic systems were introduced and investigated (in multidimensional case) in \cite{ZCh08,ZCMMP10,ZMM12a}.

\par The QGD system is simplified into the original system \eqref{eq:bgd} for $\tau=\mu=0$ and the Navier-Stokes system of equations for viscous compressible barotropic gas flow for $\tau=0$ and $\mu>0$.

\par System \eqref{eq:bgd} can be linearized on a constant solution $\rho_*\equiv\textrm{const}>0$ and $u_*=0$.
Substituting the solution in the form $\rho=\rho_*+\Delta\rho$ and $u=u_*+\Delta u$ in the equations and neglecting the terms having the second order of smallness with respect to $\Delta\rho$ and $\Delta u$ and their derivatives leads us to the following system of equations:
\begin{gather}
 \partial_t\Delta\rho+\rho_*\partial_x\Delta u=0,\ \
 \rho_*\partial_t\Delta u+p'(\rho_*)\partial_x\Delta\rho u=0.
\label{eq:bgdl}
\end{gather}
For the dimensionless unknowns $\tilde{\rho}=\frac{\Delta\rho}{\rho_*}$ and $\tilde{u}=\frac{\Delta u}{\sqrt{p'(\rho_*)}}$ we gain the acoustics system of equations:
\begin{gather}
 \partial_t\tilde{\rho}+c_*\partial_x\tilde{u}=0,\ \
 \partial_t\tilde{u}+c_*\partial_x\tilde{\rho}=0.
\label{eq:bgdln}
\end{gather}
Hereafter $c_*=\sqrt{p'(\rho_*)}$ is the background velocity of sound.
Given the initial data $\tilde{\rho}|_{t=0}=\tilde{\rho}_0$ and $\tilde{u}|_{t=0}=\tilde{u}_0$ (that one can consider complex-valued), for the solution to the last system the following energy conservation law holds
\begin{gather}
 \|\tilde{\rho}(\cdot,t)\|_{L^2(\mathbb{R})}^2+\|\tilde{u}(\cdot,t)\|_{L^2(\mathbb{R})}^2
 =\|\tilde{\rho}_0\|_{L^2(\mathbb{R})}^2+\|\tilde{u}_0\|_{L^2(\mathbb{R})}^2\ \ \text{for}\ \ t\geq 0.
\label{eq:conlaw}
\end{gather}
\par Now we pass to discretization.
Let $\omega_h$ be a uniform mesh on $\mathbb{R}$ with the nodes $x_k=kh$, $k\in\mathbb{Z}$, and step $h=X/N$.
Let $\omega^*_h$ be an auxiliary mesh with the nodes $x_{k+1/2} = (k+0.5)h$, $k\in\mathbb{Z}$.
Define a uniform mesh in $t$ with the nodes $t_m=m\Delta t$, $m\geq 0$, and step $\Delta t>0$.
We define the shift, averaging and difference quotient operators
\begin{gather*}
v_{\pm,k}=v_{k\pm 1},\ \
(sv)_{k-1/2}=\frac{v_k+v_{k+1}}{2},\ \ (\delta v)_{k-1/2}=\frac{v_k-v_{k-1}}{h},\\[1mm](\delta^\ast y)_k = \frac{y_{k+1/2}-y_{k-1/2}}{h},\ \ \delta_tv=\frac{v^{+}-v}{\Delta t},\ \ v^{+,m}=v^{m+1}.
\end{gather*}

\par We first consider a standard explicit two-level in time and three-point symmetric in space discretization of the QGD equations \eqref{eq:qgd1}-\eqref{eq:qgd3}:
\begin{gather}
\delta_t \rho + \delta^\ast j = 0,\ \
\delta_t(\rho u) + \delta^\ast\big(jsu +p(s\rho) - \Pi\big)=0\ \ \text{on}\ \ \omega_h,
\label{eq:dssw1}\\[1mm]
 j =(s\rho)su-(s\rho)w,\
 (s\rho)w =(s\tau)\big[\delta(\rho u)\big]su + (s\rho)\hat w,
\label{eq:dssw2}\\[1mm]
 (s\rho)\hat w = (s\tau)\big[(s\rho)(su)\delta u + \delta p(\rho)\big],
\label{eq:dssw3}\\[1mm]
 \Pi = \mu\delta u + (su)(s\rho)\hat w + (s\tau)\big[p'(s\rho)\big]\delta(\rho u).
\label{eq:dssw4}
\end{gather}
The main unknown functions $\rho>0$, $u$ and the parameter $\tau$ are defined on $\omega_h$ whereas $j,w,\hat{w},\Pi$ and $\mu$ are defined on $\omega_h^*$.

\par In \cite{ZMM12b} two non-standard spatial discretizations of the QGD equations \eqref{eq:qgd1}-\eqref{eq:qgd3} were constructed which are weakly conservative in energy (see their generalization to a multidimensional case in \cite{ZCMMP16}).
One of them has the ``enthalpy'' form
\begin{gather}
 \delta_t \rho + \delta^\ast j = 0,\ \
 \delta_t(\rho u) + \delta^\ast(jsu-\Pi) + s^\ast\big[(s\rho)\delta\textrm{h}(\rho)\big]=0,
\label{eq:dssw1 A}\\[1mm]
 j =s\rho\cdot su-s\rho\cdot w,\ \
 s\rho\cdot w = \big[(\tau\partial_x)_h(\rho u)\big]su + (s\rho)\hat w,\ \
\label{eq:dssw2 A}\\[1mm]
 \hat w = (s\tau)\big[(su)\delta u + \delta\textrm{h}(\rho)\big],\ \
 \Pi =\mu\delta u + (su)(s\rho)\hat w + p'(s\rho)(\tau \partial_x)_h(\rho u),
\label{eq:dssw3 A}\\[1mm]
 (\tau \partial_x)_h(\rho u) = \Big(s\frac{\tau}{\textrm{h}'(\rho)}\Big)\big\{[\delta\textrm{h}(\rho)]su + p'(s\rho)\delta u\big\},
\label{eq:dssw4 A}
\end{gather}
where $\textrm{h}(\rho)=\int_{r_0}^\rho\frac{p'(r)}{r}\,dr$, with some $r_0>0$, is the gas \textit{enthalpy} and thus $\textrm{h}'(\rho)=\frac{p'(\rho)}{\rho}$.
In the isentropic case $p(\rho)=p_1\rho^\gamma$ with $\gamma>1$, one can take $r_0=0$ and then $\textrm{h}(\rho)=\frac{\gamma}{\gamma-1}\frac{p(\rho)}{\rho}$ and $\textrm{h}'(\rho)=\gamma\frac{p(\rho)}{\rho^2}$.
Notice the non-standard $\textrm{h}(\rho)$-dependent discretizations of $\partial_xp(\rho)$ in \eqref{eq:dssw1 A} and \eqref{eq:dssw3 A} and $\tau\partial_x(\rho u)$ in \eqref{eq:dssw2 A}-\eqref{eq:dssw3 A}, see \eqref{eq:dssw4 A}.

\par We linearize scheme \eqref{eq:dssw1}-\eqref{eq:dssw4} on a constant solution $\rho_*\equiv\textrm{const}>0$ and $u_*=0$.
To do that, we write its solution in the form $\rho=\rho_*+\Delta\rho$ and $u=u_*+\Delta u$, neglect terms having the second order of smallness with respect to $\Delta\rho$ and $\Delta u$ and obtain
\begin{gather*}
 \delta_t\Delta\rho+\rho_*\delta^*s\Delta u-\tau(\rho_*)p'(\rho_*)\delta^*\delta\Delta\rho=0,
\label{eq:ds1}
\\[1mm]
 \rho_*\delta_t\Delta u+p'(\rho_*)\delta^*s\Delta\rho-\big[\mu(\rho_*)+\tau(\rho_*)\rho_*p'(\rho_*)\big]\delta^*\delta\Delta u=0.
\label{eq:ds2}
\end{gather*}
For the dimensionless unknowns $\tilde{\rho}=\frac{\Delta\rho}{\rho_*}$ and $\tilde{u}=\frac{\Delta u}{c_*}$ we get equations
\begin{gather}
 \delta_t\tilde{\rho}+c_*\delta^*s\tilde{u}-\tau(\rho_*)c_*^2\delta^*\delta\tilde{\rho}=0,
\label{eq:ds1n}
\\[1mm]
 \delta_t\tilde{u}+c_*\delta^*s\tilde{\rho}
 -\Big[\frac{\mu(\rho_*)}{\rho_*}+\tau(\rho_*)c_*^2\Big]\delta^*\delta\tilde{u}=0
\label{eq:ds2n}
\end{gather}
(cf. systems \eqref{eq:bgdl} and \eqref{eq:bgdln}).
The linearization of scheme \eqref{eq:dssw1 A}-\eqref{eq:dssw4 A} is the same.
\par Notice that since $u_*=0$, the linearization result remains the same if it would be $w=\hat{w}$ and the terms dependent on $u$ were omitted in the definition of these variables, i.e., for example, $w=\hat{w}=\frac{\tilde{\tau}}{s\rho}\delta p(\rho)$ instead of formulas in \eqref{eq:dssw2}-\eqref{eq:dssw3}.

\par We assume that the regularization parameter and viscosity coefficient are given by usual QGD-formulas
\begin{equation*}
 \tau(\rho)= \frac{\alpha h}{\sqrt{p'(\rho)}},\ \
 \mu(\rho) = \alpha_s\tau(\rho)\rho p'(\rho),
\end{equation*}
where $\alpha>0$ and $\alpha_s\geq 0$ are parameters.
Then omitting tildes above $\rho$ and $u$, equations \eqref{eq:ds1n}-\eqref{eq:ds2n} can be rewritten in the following recurrent form
\begin{gather}
 \rho^{+}=\rho-\frac{\beta}{2}(u_+-u_-)+\alpha \beta(\rho_--2\rho+\rho_-),
\label{eq:ds1e}\\[1mm]
 u^{+}=u-\frac{\beta}{2}(\rho_+-\rho_-)+\varkappa\alpha\beta(u_+-2u+ u_-)
\label{eq:ds2e}
\end{gather}
with three parameters $\alpha$, $\beta:=c_*\frac{\Delta t}{h}$ and $\varkappa:=\alpha_s+1\geq 1$.
The functions $\rho^0$ and $u^0$ are given, i.e., we consider the initial-value problem for the scheme.
Below it is convenient to consider $\rho$ and $u$ as complex-valued mesh functions.
\section{\large Weak conservativeness analysis}
Let $\mathbf{y}^m=(\rho^m\ u^m)^T$, $m\geq 0$, be a column-vector function on $\omega_h$
and the linearized difference scheme \eqref{eq:ds1e}-\eqref{eq:ds2e} be rewritten in a matrix form
\begin{gather}
\mathbf{y}^{+}
= \begin{pmatrix}
    \alpha\beta & \frac{\beta}{2}
\\[1mm]
\frac{\beta}{2} & \varkappa\alpha\beta
\end{pmatrix}
\mathbf y_-
+\begin{pmatrix}
1-2\alpha\beta&0
\\[1mm]
0&1-2\varkappa\alpha\beta
\end{pmatrix}\mathbf{y}
+
\begin{pmatrix}
\alpha\beta & -\frac{\beta}{2}
\\[1mm]
-\frac{\beta}{2} & \varkappa\alpha\beta
\end{pmatrix}
\mathbf{y}_+.
\label{eq:dsm}
\end{gather}
\par Let $H$ be a Hilbert space of complex valued square-summable on $\omega_h$ vector functions, i.e. having a finite norm
\[
 \|\mathbf{y}\|_H=\Big(h\sum_{k=-\infty}^\infty|\mathbf{y}_k|^2\Big)^{1/2}.
\]
\par For $\mathbf{y}^0=(\rho^0\ u^0)^T\in H$ we have that $\mathbf{y}^m\in H$ for all $m\geq 1$.
We define a \textit{weak conservativeness} of scheme \eqref{eq:dsm} as validity of the bound
\begin{gather}
 \sup_{m\geq 0}\|\mathbf{y}^m\|_H\leq \|\mathbf{y}^0\|_H\ \ \forall\mathbf{y}^0\in H.
\label{eq:stab1}
\end{gather}
This definition is motivated by the energy conservation law \eqref{eq:conlaw} for the acoustics system of equation \eqref{eq:bgdln}.
It is essential to notice that for the linearized QGD-system \eqref{eq:qgd1}-\eqref{eq:qgd3} namely the corresponding inequality holds in place of equality \eqref{eq:conlaw} so that it is natural to study the bound for schemes based on such a system.
Of course, estimate \eqref{eq:stab1} guarantees the uniform in time stability in $H$ with respect to initial data.
\par We first substitute a partial solution in the form $\mathbf{y}_k^m=e^{\mathbf{i}k\xi}\mathbf{v}^m(\xi)$, $k\in\mathbb{Z}$, $m\geq 0$, where $\mathbf{i}$ is the imaginary unit and $0\leq\xi\leq 2\pi$ is a parameter, into \eqref{eq:dsm} and obtain
\begin{gather}
 \mathbf{v}^{+}(\xi)=G(\xi)\mathbf{v}(\xi),\ \
G(\xi)
 =\begin{pmatrix}
1-\omega_1 & -\mathbf{i}\omega_2
\\[1mm]
-\mathbf{i}\omega_2 & 1-\varkappa\omega_1
\end{pmatrix},
\label{eq:stab3}
\end{gather}
where we denote
$\omega_1=4\alpha\beta\theta$, $\theta=\sin^2\frac{\xi}{2}\in[0,1]$ and $\omega_2=\beta\sin\xi$ for brevity.
Below it is important that $\omega_2^2=4\beta^2\theta(1-\theta)$.

\par It is known  (see similar formulas in \cite{GR73}) that if $\mathbf{y}^0=(\rho^0\ u^0)^T\in H$, then there exists a function
$\mathbf{v}^0\in L^2(0,2\pi)$ such that
\[
 \mathbf{v}^0(\xi)=\frac{1}{\sqrt{2\pi}}\sum_{k=-\infty}^\infty\mathbf{y}_k^0 e^{-\mathbf{i}k\xi},
\]
and we can write the solution to scheme \eqref{eq:dsm} in an integral form
\[
 \mathbf{y}_k^m=\frac{1}{\sqrt{2\pi}}\int_0^{2\pi}\mathbf{v}^m(\xi)e^{\mathbf{i}k\xi}\,d\xi,\ \ k\in\mathbb{Z},
\]
where $\mathbf{v}^m\in L^2(0,2\pi)$ due to \eqref{eq:stab3}.
The following Parseval identity also holds
\begin{gather}
 \|\mathbf{y}^m\|_H=\sqrt{h}\,\|\mathbf{v}^m\|_{L^2(0,2\pi)},\ \ m\geq 0.
\label{eq:stab5}
\end{gather}
\par The von Neumann type spectral condition
\begin{gather}
 \max_{0\leq\xi\leq 2\pi}\max_l\big|\lambda_l\big(G(\xi)\big)\big|\leq 1
\label{eq:stab7}
\end{gather}
is known to be \textit{a necessary condition} for property \eqref{eq:stab1} to hold (see similar result in  \cite{GR73}).
Hereafter $\lambda_l(A)$ are eigenvalues of a matrix $A$.
\par Let us determine the spectral form of property \eqref{eq:stab1}.
\begin{lemma}
\label{lem1}
Validity of the spectral bound
\begin{gather}
 \max_{0\leq\xi\leq 2\pi}\max_l\lambda_l\big((G^*G)(\xi)\big)\leq 1
\label{eq:stab9}
\end{gather}
is necessary and sufficient for the weak conservativeness property \eqref{eq:stab1} to hold.
\end{lemma}
\begin{proof}
\par Due to the Parseval identity \eqref{eq:stab5} and formula \eqref{eq:stab3} we have
\[
 h^{-1}\|\mathbf{\hat{y}}\|_H^2=\|\mathbf{\hat{v}}\|_{L^2(0,2\pi)}^2=\|G\mathbf{v}\|_{L^2(0,2\pi)}^2
 =(G^*G\mathbf{v},\mathbf{v})_{L^2(0,2\pi)}.
\]
Since $(G^*G)(\xi)\geq 0$ is a Hermitian matrix, it has a spectral decomposition $(G^*G)(\xi)=U^*(\xi)\Lambda(\xi)U(\xi)$, where $U(\xi)$ is a unitary matrix, and $\Lambda(\xi)$ is a diagonal matrix with numbers
$\lambda_l\big((G^*G)(\xi)\big)\geq 0$ forming its diagonal.
Hence for $z(\xi):=U(\xi)\mathbf{v}(\xi)$ we have
\[
 (G^*G\mathbf{v},\mathbf{v})_{L^2(0,2\pi)}=(\Lambda\mathbf{z},\mathbf{z})_{L^2(0,2\pi)}=\|\Lambda^{1/2}\mathbf{z}\|_{L^2(0,2\pi)}^2.
\]
Thus $\|\mathbf{y}^m\|_H^2=h\|\Lambda^{m/2}\mathbf{z}^0\|_{L^2(0,2\pi)}^2$ for $m\geq 0$, and bound \eqref{eq:stab1} is equivalent to the following one
\[
 \sup_{m\geq 0}\|\Lambda^{m/2}\mathbf{z}^0\|_{L^2(0,2\pi)}^2\leq\|\mathbf{z}^0\|_{L^2(0,2\pi)}^2\ \ \forall\mathbf{z}^0\in L^2(0,2\pi).
\]
It holds if and only if the spectral bound \eqref{eq:stab9} holds.
\end{proof}
\begin{remark}
Under validity of the spectral bound \eqref{eq:stab9}, the norm $\|\mathbf{y}^m\|_H$ is actually non-increasing in $m\geq 0$ that serves as a stronger property than \eqref{eq:stab1}.
\end{remark}
\par In the proof of this lemma, the specific form and dimension of the matrix $G$ are clearly inessential, and actually it holds in general case.
\par In our case the matrix $G^*G$ has the form
\begin{equation*}
 G^*G=
\begin{pmatrix}
 (1-\omega_1)^2+\omega_2^2      & -\mathbf{i}(1-\varkappa)\omega_1\omega_2
\\[1mm]
 \mathbf{i}(1-\varkappa)\omega_1\omega_2 & (1-\varkappa\omega_1)^2 + \omega_2^2
\end{pmatrix}.
\end{equation*}
Note that $G^{\ast}G=\big[(1-\omega_1)^2+\omega_2^2\big]I$ in the simplest case $\varkappa=1$, where $I$ is a unit matrix.
\begin{theorem}
\label{th:1}
The necessary spectral condition \eqref{eq:stab7} holds if and only if
\begin{equation}
 \beta\leq\min\Big\{(\varkappa+1)\alpha,\frac{1}{2\varkappa\alpha}\Big\}.
\label{eq:nsc}
\end{equation}
\end{theorem}
\begin{proof}
The characteristic polynomial for the matrix $G$ has the following form
\begin{multline}
 q_1(\lambda)=\lambda^2-(\tr G)\lambda+\det G
 =\lambda^2+ [(\varkappa+1)\omega_1-2]\lambda+\\[1mm]+\big[\varkappa\omega_1^2+\omega_2^2+1-(\varkappa + 1)\omega_1\big].
\label{eq:nsc1}
\end{multline}
We set $a_0:=q_1(1)=1-(\tr G)+\det G=\varkappa\omega_1^2+\omega_2^2\geq 0$ and notice that $a_0=q_1(1)=0$ if and only if $G=I$.
We transform the unit circle $\{|\lambda|\leq 1\}$ with a punctured point $(1,0)$ on $\mathbb{C}$ into the closed left half-plane  $\{\Rea z\leq 0\}$ and put
\begin{equation*}
 \hat{q}_1(z):=(z-1)^2q_1\Big(\frac{z+1}{z-1}\Big)
 =a_0z^2+2a_1z+a_2,
\end{equation*}
where $a_1=1-\det G$, $a_2=1+\tr G+\det G$.
It is well known that for $a_0>0$ the roots $\hat{q}_1(z)$ lie in $\{\Rea z\leq 0\}$ under the conditions $a_1\geq 0$ and $a_2\geq 0$, i.e.
\[
 \varkappa\omega_1^2+\omega_2^2-(\varkappa+1)\omega_1\leq 0,
\ \
 \varkappa\omega_1^2+\omega_2^2-2(\varkappa+1)\omega_1+4\geq 0.
\]
We rewrite these conditions as
\begin{gather}
 \beta(4\varkappa\alpha^2\theta+1-\theta)-(\varkappa+1)\alpha\leq 0\ \ \text{for}\ \ 0\leq\theta\leq 1,
\label{eq:nsc5}\\[1mm]
 r(\theta):=\beta^2(4\varkappa\alpha^2-1)\theta^2-\beta\big(2(\varkappa+1)\alpha-\beta\big)\theta+1\geq 0\ \ \text{for}\ \
 0\leq\theta\leq 1.
\label{eq:nsc7}
\end{gather}
\par The left-hand side of \eqref{eq:nsc5} is linear in $\theta$, thus it suffices to test it for $\theta=0,1$ that leads us to the condition
\begin{gather}
 \beta\leq\min\Big\{(\varkappa+1)\alpha,\frac{\varkappa+1}{4\varkappa\alpha}\Big\}.
\label{eq:nsc9}
\end{gather}
\par Next we analyze condition \eqref{eq:nsc7}.
Notice that $r(0)=1$ and due to \eqref{eq:nsc9} we have $2(\varkappa+1)\alpha-\beta>0$.
For $a=4\varkappa\alpha^2-1\neq 0$ the vertex of the parabola $r(\theta)$ is given by
\begin{equation*}
 \theta_v=\frac{2(\varkappa+1)\alpha-\beta}{2\beta(4\varkappa\alpha^2-1)}.
\end{equation*}
\par For $4\varkappa\alpha^2-1>0$ the property $\theta_v>1$ means that
$\beta<\frac{(\varkappa+1)\alpha}{4\varkappa\alpha^2-0.5}$,
and it holds due to \eqref{eq:nsc9}.
Hence condition \eqref{eq:nsc7} reduces to $r(1)\geq 0$, i.e.
\begin{equation}
 4\varkappa\alpha^2\beta^2-2(\varkappa+1)\alpha\beta+1
 =4\varkappa\alpha^2\Big(\beta-\frac{1}{2\varkappa\alpha}\Big)\Big(\beta-\frac{1}{2\alpha}\Big)\geq 0.
\label{eq:r10}
\end{equation}
For $4\varkappa\alpha^2-1<0$ we have $\theta_v<0$, so that condition \eqref{eq:nsc7} reduces again to $r(1)\geq 0$.
For  $4\varkappa\alpha^2-1=0$ the condition also reduces to $r(1)\geq 0$ (since $r(0)=1$).
\par Since $\varkappa\geq 1$, inequality \eqref{eq:r10} means that either of the conditions
\begin{equation}
\beta\leq\frac{1}{2\varkappa\alpha},\ \ \beta\geq\frac{1}{2\alpha}
\label{eq:nsc11}
\end{equation}
holds.
Combining them with \eqref{eq:nsc9}, we obtain \eqref{eq:nsc}.
\end{proof}
Now we turn to the spectral criterion \eqref{eq:stab9}.
\begin{theorem}
\label{th:2}
The spectral criterion \eqref{eq:stab9} holds if and only if
\begin{gather}
  \beta\leq \min\Big\{2\alpha,\frac{1}{2\varkappa\alpha}\Big\}.
\label{eq:nssc}
\end{gather}
\end{theorem}
\begin{proof}
The characteristic polynomial of $G^*G$ has the following form
\[
 q_2(\lambda)=\lambda^2-\tr(G^*G)\lambda+(\det G)^2.
\]
Since $\lambda_l(G^*G)\geq 0$,
the property $|\lambda_l(G^*G)|\leq 1$ means validity of the conditions
\begin{gather}
 \frac{1}{2}\tr(G^*G)\leq 1,\ \ q_2(1)=1-\tr(G^*G)+(\det G)^2\geq 0.
\label{eq:nssc5}
\end{gather}

\par The first of them has the form
\[
 \frac{\varkappa^2+1}{2}\,\omega_1^2+\omega_2^2-(\varkappa+1)\omega_1\leq 0
\]
and can be specified as
\begin{gather}
 8\alpha^2\beta^2(\varkappa^2+1)\theta^2 + 4\beta^2\theta(1-\theta)-4\alpha\beta(\varkappa + 1)\theta\leq 0\ \
 \text{for}\ \ 0\leq\theta\leq 1.
\label{eq:nssc7}
\end{gather}
After dividing by $4\beta\theta$ we get that it suffices to confine ourselves with the values $\theta=0,1$ that leads to the condition
\begin{gather}
\beta\leq\min\Big\{(\varkappa + 1)\alpha,\frac{\varkappa + 1}{2(\varkappa^2+1)\alpha}\Big\}.
\label{eq:nssc8}
\end{gather}

\par In order to transform the second condition \eqref{eq:nssc5} we notice that
\[
 \tr(G^*G)=2(b+1)+(\varkappa-1)^2\omega_1^2,\ \
 (\det G)^2=(b+1)^2,
\]
where $b:=\varkappa\omega_1^2+\omega_2^2-(\varkappa+1)\omega_1$, see \eqref{eq:nsc1}.
Hence the following factorization holds
\[
 q_2(1)=b^2-(\varkappa-1)^2\omega_1^2=\big(b-(\varkappa-1)\omega_1\big)\big(b+(\varkappa-1)\omega_1\big),
\]
that is decisive for the simplicity of our analysis.
Since $\varkappa\geq 1$, the condition $q_2(1)\geq 0$  is equivalent to validity of either of the conditions
\[
 \varkappa\omega_1^2+\omega_2^2-2\omega_1\leq 0,\ \ \varkappa\omega_1^2+\omega_2^2-2\varkappa\omega_1\geq 0,
\]
i.e., more specifically, to validity of either of the conditions
\begin{gather*}
 \beta(4\varkappa\alpha^2\theta+1-\theta)-2\alpha\leq 0\ \ \text{for}\ \ 0\leq\theta\leq 1,
\\[1mm]
 \beta(4\varkappa\alpha^2\theta+1-\theta)-2\varkappa\alpha\geq 0\ \ \text{for}\ \ 0\leq\theta\leq 1.
\label{eq:nssc9}
\end{gather*}
As above they respectively mean that the inequalities
\begin{gather}
 \beta\leq\min\Big\{2\alpha,\frac{1}{2\varkappa\alpha}\Big\},\ \
 \beta\geq\max\Big\{2\varkappa\alpha,\frac{1}{2\alpha}\Big\}
\label{eq:nssc11}
\end{gather}
hold. Combining them with  \eqref{eq:nssc8} we complete the proof.
\end{proof}

\par It is essential that the function on the right-hand side of condition \eqref{eq:nssc} reaches its maximal value at
$\alpha=\alpha_\ast:=\frac{1}{2\sqrt{\varkappa}}\leq\frac12$
and the maximal value equals $\frac{1}{\sqrt{\varkappa}}\leq 1$.
Hence the criterion coincides with the standard CFL stability condition $\beta\leq 1$ if and only if $\alpha=\alpha_*$ and $\varkappa=1$.
The criterion gives an important information on the optimal choice of $\alpha$ since in practice for the original non-linear problem $\alpha$ is normally sought experimentally.
Note also that criterion \eqref{eq:nssc} and the necessary condition \eqref{eq:nsc} coincide only in the case $\alpha\geq\alpha_*$.
\par We call attention to a paradoxical moment: criterion \eqref{eq:nssc} becomes \textit{stronger} as the coefficient of  ``effective viscosity'' $\varkappa$ increases (it is harder to say that about the necessary condition \eqref{eq:nsc}).
Therefore the best choice in the present bounds is $\alpha_s=0$, i.e. $\varkappa=1$.
But this conclusion is not universal in practice and it is known that in some situations $\alpha_s>0$ has to be taken (see, for example, \cite{ZG16}).
\par In Fig. \ref{ris:crit} we compare the necessary condition, the criterion of stability and the sufficient condition as well as the results of numerical experiments for the original system \eqref{eq:bgd} for $p(\rho)=\rho^2$ (the scaled case of the shallow water equations)
and $\varkappa=\frac{7}{3}$.
The sufficient condition was obtained in \cite{SSh13} only for these $p(\rho)$ and $\varkappa$ by the energy method and has the form
\[
\beta\leq\min\left\{\frac{2\alpha}{1+6\alpha+4\alpha^2},\frac{4\alpha}{1+6\alpha+16\alpha^2}\right\}.
\]
Notice that the first fraction in it is less than the second one for $0<\alpha<\frac{3+\sqrt{17}}{8}\approx 0.890$.
The corresponding graph is almost flat for $0.3\leq\alpha\leq 0.9$ in contrast to the cases of necessary condition and criterion.
The computations are accomplished for $0\leq t\leq 0.5$ for the Riemann problem with the discontinuous initial data
\begin{equation*}
\rho_0(x) =
\begin{cases}
 1, & x < 0\\
 0.1, & x > 0
\end{cases},\ \
u_0(x) =
\begin{cases}
 0.1, & x < 0,\\
 0, & x > 0
\end{cases}
\end{equation*}
for both schemes  \eqref{eq:dssw1}-\eqref{eq:dssw4} and \eqref{eq:dssw1 A}-\eqref{eq:dssw4 A} with $h=1/125$.
\begin{figure}[h]
\begin{minipage}[h]{0.49\linewidth}
\center{\includegraphics[width=1\linewidth,height=0.2\textheight]{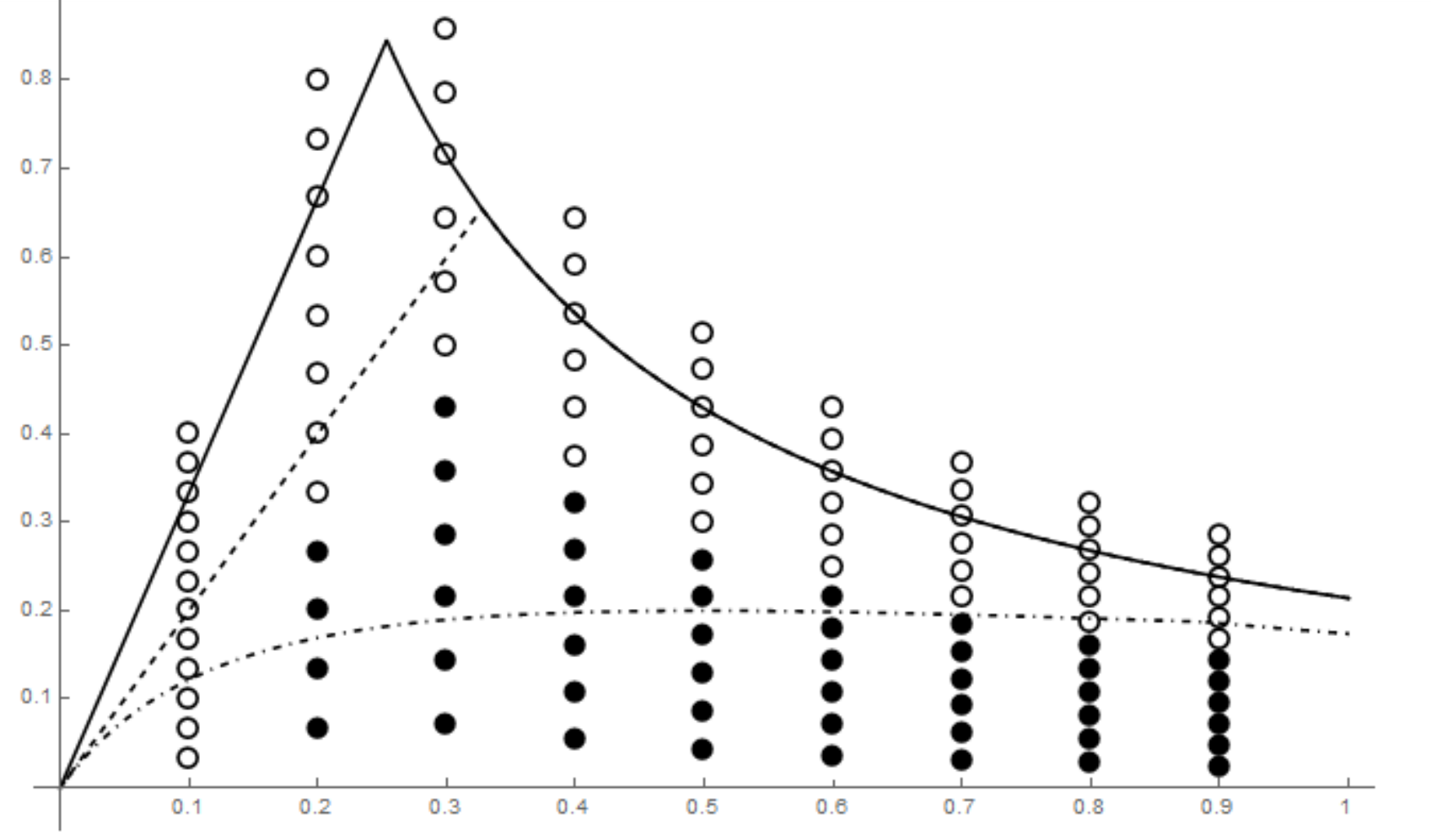}%eps}
\\ (a) The standard scheme}
\end{minipage}
\hfill
\begin{minipage}[h]{0.49\linewidth}
\center{\includegraphics[width=1\linewidth,height=0.2\textheight]{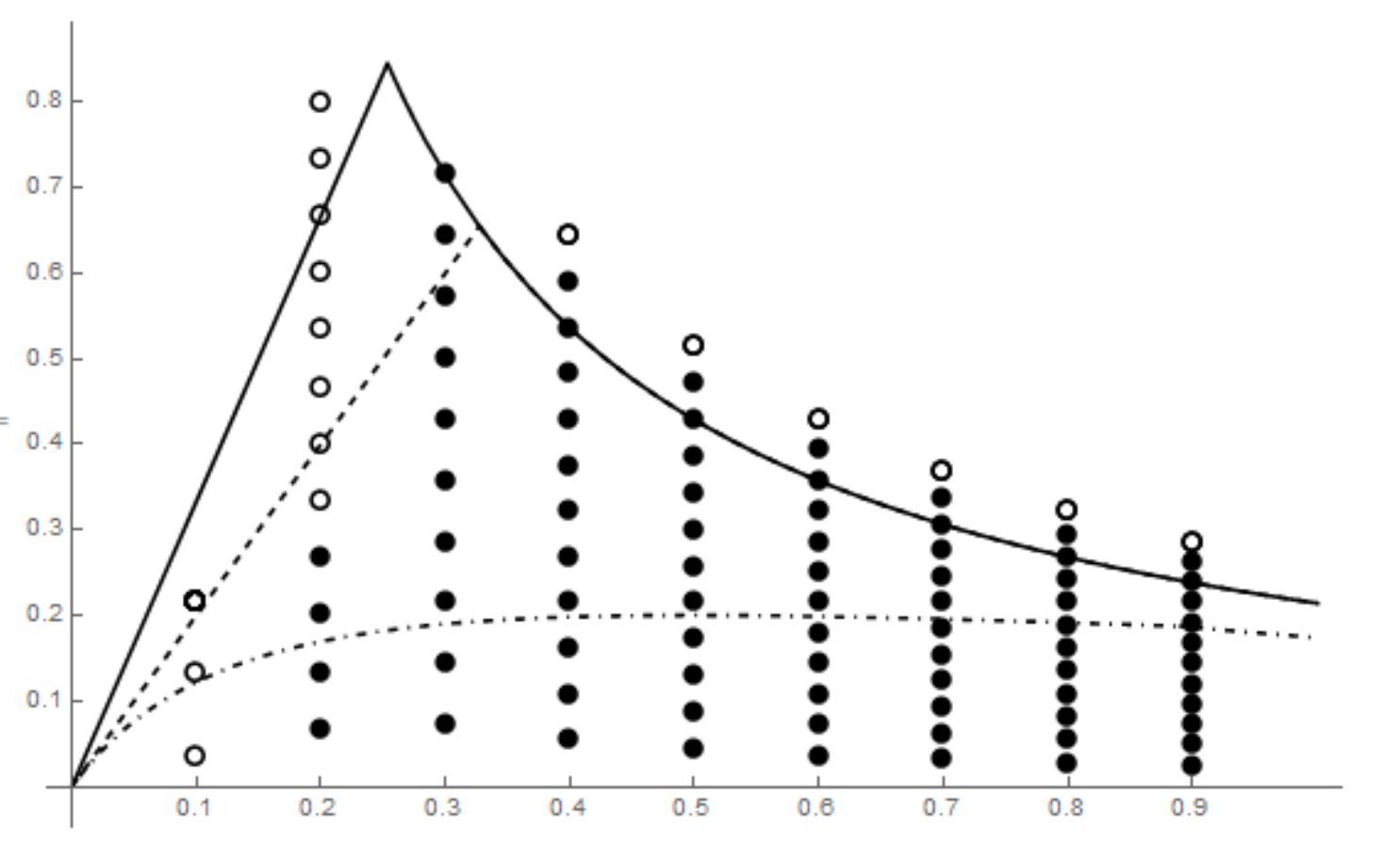}%eps}
\\ (b) The ``enthalpy'' scheme}
\end{minipage}
\caption{{\small{The weak conservativeness analysis: the necessary condition (solid line), the criterion (dash line), the sufficient condition (dotdash line) together with conservative (painted balls) and non-conservative (unpainted balls) computations for the Riemann problem in dependence with $\alpha$}}}
\label{ris:crit}
\end{figure}

\par We observe a good correspondence of the obtained criterion with the experimental results, and that the sufficient condition underestimates the criterion up to several times in the most interesting region $\alpha\approx\alpha_*$.
Also the results for the ``enthalpy'' scheme \eqref{eq:dssw1 A}-\eqref{eq:dssw4 A} are clearly different from and better than for the standard one \eqref{eq:dssw1}-\eqref{eq:dssw4} though the above linearized analysis gives the same results for them.

\par We have identified non-conservative computations by noticeable well-known oscillations of the numerical solutions (some of  computations have not even been completed due to overflow).
In Fig. \ref{ris:con} we give an example of conservative and non-conservative solutions $\rho$ and $u$
for the ``enthalpy'' scheme (at time $t=0.5$) for $\alpha =0.4$ and two neighboring values of $\beta$ from Fig. \ref{ris:crit} (b).

\begin{figure}[h]
\begin{minipage}[h]{0.49\linewidth}
\center{\includegraphics[width=1\linewidth]{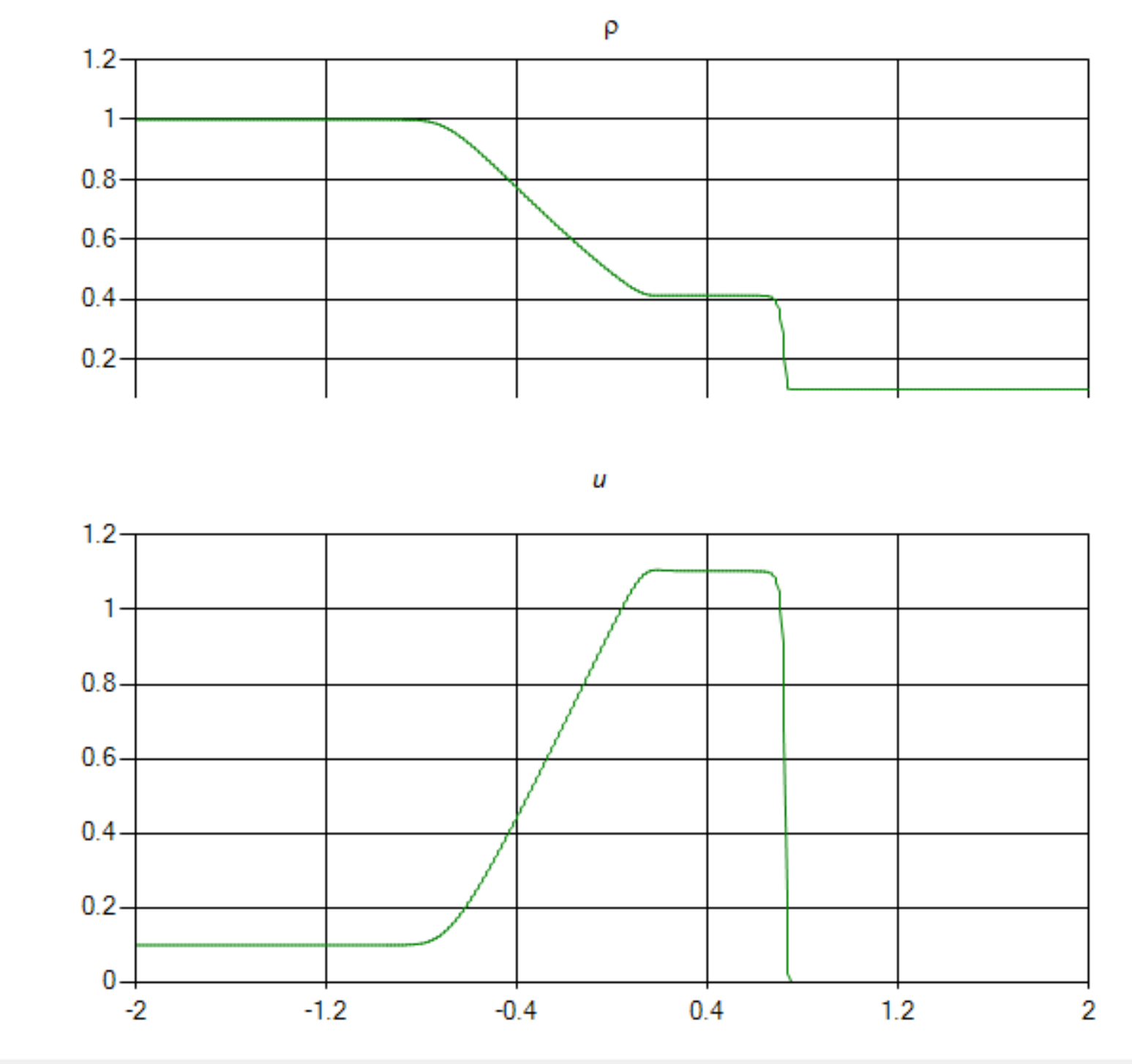}%eps}
\\ (a) The conservative solution, $\beta\approx 0.589$}
\end{minipage}
\hfill
\begin{minipage}[h]{0.49\linewidth}
\center{\includegraphics[width=1\linewidth]{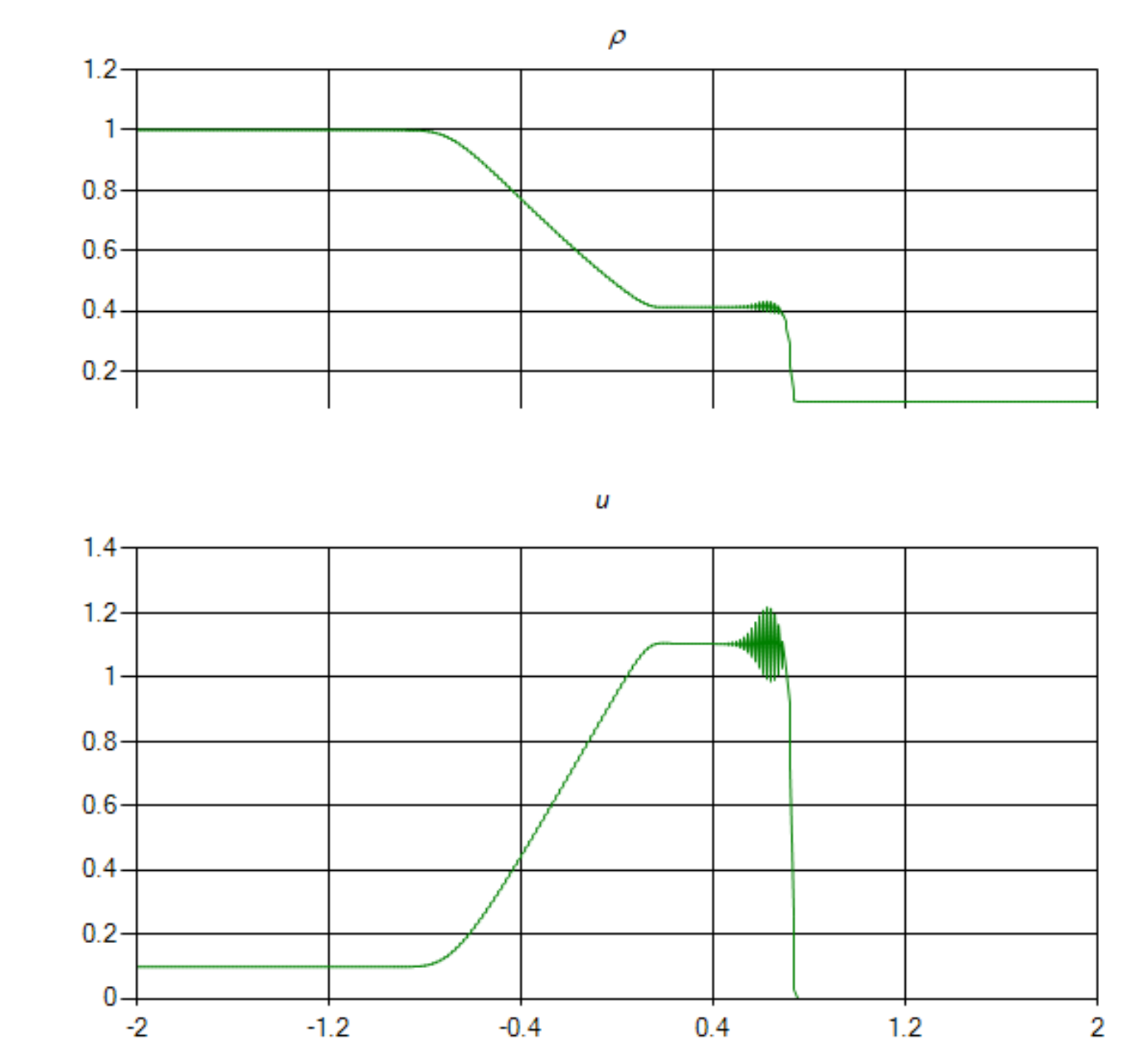}%eps}
\\ (b) The non-conservative solution, $\beta\approx 0.643$}
\end{minipage}
\caption{{\small{The examples of conservative and non-conservative solutions for the ``enthalpy``scheme for $\alpha =0.4$ (at time $t=0.5$)}}}
\label{ris:con}
\end{figure}

\section{\large The case of the schemes based on a simplified regularization}

We also consider a simplified (quasi-hydrodynamic \cite{E07,Sh09,Z08}) regularization \eqref{eq:qgd1}-\eqref{eq:qgd3}, where the terms with $\partial_x(\rho u)$ are omitted, in particularly, it becomes $w=\hat{w}$.
Correspondingly in schemes \eqref{eq:dssw1}-\eqref{eq:dssw4} and \eqref{eq:dssw1 A}-\eqref{eq:dssw4 A} we have to omit both terms with respectively $\delta(\rho u)$ and $(\tau \partial_x)_h(\rho u)$.
In the linearized scheme the term $\tau(\rho_*)c_*^2$ disappears from equation \eqref{eq:ds2n}, hence now $\varkappa=\alpha_s$.
Notice that usually $0<\alpha_s\leq 1$ though in specific cases $\alpha_s>1$ can be also taken.
\begin{theorem}
For the simplified scheme based on the quasi-hydrodynamic regularization the following results are valid:
\par (1) in the case $0\leq\alpha_s\leq 1$
the necessary condition \eqref{eq:stab7} and criterion \eqref{eq:stab9} hold if and only if respectively
\begin{gather}
 \beta\leq\min\Big\{(\alpha_s+1)\alpha,\frac{1}{2\alpha}\Big\},
\label{eq:nsc s}\\[1mm]
  \beta\leq \min\Big\{2\alpha_s\alpha,\frac{1}{2\alpha}\Big\};
\label{eq:nssc s}
\end{gather}

\par (2) in the case $\alpha_s\geq 1$ the results of Theorems \ref{th:1} and \ref{th:2} remain valid with $\varkappa=\alpha_s$.
\end{theorem}
\begin{proof}
The above given analysis holds true except for some changes in the case $0\leq\varkappa=\alpha_s\leq 1$.
In this case, inequalities \eqref{eq:nsc11} are replaced by the following ones
\[
\beta\leq\frac{1}{2\alpha},\ \ \beta\geq\frac{1}{2\varkappa\alpha};
\]
they being combined with \eqref{eq:nsc9} lead to \eqref{eq:nsc s}.
Also inequalities \eqref{eq:nssc11} are replaced by the following ones
\[
 \beta\leq\min\Big\{2\varkappa\alpha,\frac{1}{2\alpha}\Big\},\ \
 \beta\geq\max\Big\{2\alpha,\frac{1}{2\varkappa\alpha}\Big\},
\]
they being combined with \eqref{eq:nssc8} lead to \eqref{eq:nssc s}.
\end{proof}
The maximal value of the function on the right-hand side of criterion \eqref{eq:nssc s} is reached at
$\alpha=\alpha_*:=\frac{1}{2\sqrt{\alpha_s}}\geq\frac12$
and equals $\sqrt{\alpha_s}\leq 1$.
We notice that the necessary condition \eqref{eq:nsc s} is especially rough compared to criterion \eqref{eq:nssc s} for $\alpha_s\approx 0$, including the case $\alpha_s=0$ when actually the stability is absent at all.

\begin{acknowledgement}
The study was partially supported by the RFBR, project nos. 16-01-00048 and 18-01-00587.
\end{acknowledgement}

\end{document}